\documentclass{amsart}
\usepackage{graphicx} 
\usepackage{bm,amsmath,amssymb,amsfonts,dsfont,mathrsfs}
\usepackage{subcaption}
\usepackage{caption}
\usepackage[utf8]{inputenc}
\usepackage[T1]{fontenc} 
\usepackage{tikz,lipsum,lmodern}
\usepackage[most]{tcolorbox}
\usepackage{float}

\def\Ne{\mathds{N}}

\def\Re{\mathds{R}}

\def\1e{\mathds{1}}

\def\vg0{\mathbf{0}}

\def\D{\mathscr{D}}

\def\eqd{: =}

\def\segoc#1#2{(#1, #2]}

\def\set#1#2{
    \left\{#1
    \left|\vphantom{#1#2}\right.
    #2\right\}
    }

\def\mod#1{
    \left|#1
    \right|
    }

\def\norm#1{\left\|#1\right\|}

\def\argmin{\mathop{\mathrm{argmin}}}

\def\ud{\,\mathrm{d}}

\title{Multi-resolution deconvolution}
\author[Karamehmedovi\'c, Mar\'echal, Car\o e, and Baalbaki]{Mirza Karamehmedovi\'c, Pierre Mar\'echal, Martin S\ae bye Car\o e, and Lara Baalbaki}

\newcommand{\cb}[1]{\langle#1\rangle}

\newcommand{\xxi}{\langle\xi\rangle}

\newtheorem{theorem}{Theorem}

\newtheorem{remark}{Remark}

\newtheorem{assumption}{Assumption}

\begin{document}

\maketitle

\begin{abstract}
    We extend the classical deconvolution framework in $\Re^n$ to the case with a pseudodifferential-like solution operator with a symbol depending on both the base and cotangent variable. Our framework enables deconvolution with spatially varying resolution while maintaining a set global stability, and it additionally allows rather general distributional convolution kernels. We provide consistency, convergence and stability results, as well as convergence rates. Finally, we include numerical examples supporting our results and demonstrating advantages of the generalized framework.
\end{abstract}

\section{Introduction}\label{section:introduction}

The deconvolution problem is at the heart of a wide variety
of problems in applied sciences, ranging from signal and image
processing to statistics. This is an ill-posed inverse problem,
the regularization of which has been the subject of numerous works.
In this article, we are interested in the so-called mollification
approach, with the aim of giving a {\sl variable resolution}
version of this approach. Throughout our work, 'resolution' means the amount of accurate spectral content recovered by an inversion method. We generally think of resolution as a local quality of a solution of the inverse problem.

There are basically three forms of mollification: the first one
consists in {\sl mollifying} the data before applying the theoretical
inverse (see ~\cite{murio2011mollification} and the references therein).
The second, introduced in the seminal article~\cite{louis1990mollifier},
uses the Hilbertian duality to replace the initial
problem by the adjoint equation, which sometimes admits an explicit solution.
This approach was subsequently referred to as the method
of approximate inverses and was further developed in~\cite{schuster2007method}.
The third form of mollification is said to be
{\sl variational}, because it is primarily formulated
in terms of optimization, just as Tikhonov regularization.
To the best of our knowledge,
it originates in~\cite{lannes1987stabilized}.
It has been explored subsequently
in~\cite{alibaud2009variational,bonnefond2009variational,hohage2022mollifier}.

It should be noted that, as expressed in~\cite{hohage2022mollifier}, the first two forms
of mollification lead to the same solution when dealing with the deconvolution
problem. We mostly focus on variational mollification, which turns
out to be more flexible and at the same time reaches
order optimal convergence rates~\cite{hohage2022mollifier}.

We consider the equation $T_\gamma f=g$,
in which $T_\gamma\colon L^2(\Re^n)\to L^2(\Re^n)$ is the convolution
operator by $\gamma$: $T_\gamma f=\gamma\ast f$, $f$ is the unknown function, and $g$ is the data. Here,
$n\in\Ne^*=\{1,2,\dots\}$, and
the convolution kernel
$\gamma$ is (temporarily) taken in $L^1(\Re^n)$,
with $\widehat\gamma(0)\neq0$, $\widehat\gamma$ being
the Fourier transform of~$\gamma$.
Throughout we shall use the following convention
for the Fourier transform and its inverse:
\[
Fu(\xi)=\widehat{u}(\xi)=\int_{x\in\Re^n}e^{-2\pi ix\xi}u(x)\ud x,
\quad F^{-1}u(x)=
\int_{\xi\in\Re^n}e^{2\pi ix\xi}u(\xi)\ud\xi.
\]
Here, $x\xi$ denotes the Euclidean scalar
product of~$x$ and~$\xi$ in $\Re^n$.

For the deconvolution problem of finding $f$ given $g=T_{\gamma}f=\gamma\ast f$, we here assume the standard setup $f,g\in L^1(\Re^n)\cap L^2(\Re^n)$~\cite{hohage2022mollifier}. Given a linear ill-posed operator equation $Tf=g$,
a family of operators $(R_\alpha)_{\alpha\in\segoc{0}{1}}$
is said to be a
{\sl regularization} of the pseudo-inverse~$T^\dagger$ if,
for all $g$ in the domain $\D(T^\dagger)$ of~$T^\dagger$, there is
a {\sl parameter choice rule} $\alpha(\delta,g^{(\delta)})$
such that
\begin{enumerate}
\item[(i)]
$\sup\set{\mod{\alpha(\delta,g^{(\delta)})}}
{g^{(\delta)}\in G,\;\norm{g^{(\delta)}-g}\leq\delta}\to 0$
as $\delta\searrow 0$;
\item[(ii)]
$\sup\set
{\norm{R_{\alpha(\delta,g^{(\delta)})}g^{(\delta)}-T^\dagger g}}
{g^{(\delta)}\in G,\;\norm{g^{(\delta)}-g}\leq\delta}\rightarrow0\,\,\,{\rm as}\,\,\,\delta\searrow0$.
\end{enumerate}
See~\cite{engl1996regularization}, Definition 3.1.
Recall that if the family $(R_\alpha)$ converges pointwise
to~$T^\dagger$ on $\D(T^\dagger)$, 
a property that we call {\sl consistency},
then it is a regularization of~$T^\dagger$.
See again~\cite{engl1996regularization},
Proposition 3.4.

Mollification consists in aiming for the
reconstruction of $\phi_\beta\ast f$, where
$(\phi_\beta)_{\beta\in\segoc{0}{1}}$ is a
{\sl mollifier}, which we now specify.
Given an integrable function~$\phi$
and such that $\widehat{\phi}(\xi)=1$ if and only if
$\xi=0$, we let
\begin{equation*}
\phi_\beta(x)\eqd
\frac{1}{\beta^n}\phi\left(\frac{x}{\beta}\right).
\end{equation*}
Recall that $\widehat{\phi}_\beta(\xi)=\widehat{\phi}(\beta\xi)$,
so that~$\beta$ acts as a {\sl dilation parameter} in the
Fourier domain.
We sometimes call $\phi_\beta$ a {\sl deconvolution kernel}.
Its choice is essentially guided by the {\sl interpretability}
of the solution, which leaves a certain flexibility. 
Assumptions such as isotropy, smoothness, compactness
of the support, are usual (but non-mandatory) assumptions.

The solution obtained via both the original form
of mollification and the approximate inverses reads,
in the Fourier domain,
\begin{equation}
\label{approx-inv-solution}
\frac{\widehat{\phi}_\beta}{\widehat{\gamma}}\widehat{g},
\end{equation}
(see~\cite[Appendix B]{hohage2022mollifier})
while a variational form of mollification
gives the solution
\begin{equation}
\label{variational-mollif}
\widehat{f}_\beta=
\frac{\overline{\widehat{\gamma}}\widehat{\phi}_\beta}
{\mod{\widehat{\gamma}}^2+\mod{1-\widehat{\phi}_\beta}^2}
\widehat{g}.
\end{equation}
The inverse Fourier transform of the 
latter function is, indeed, the unique minimizer of
the functional
\begin{equation*}
f\mapsto\norm{T_\gamma f-C_{\beta}g}^2+\norm{(I-C_\beta)f}^2,
\end{equation*}
in which~$I$ denotes the identity and~$C_\beta$
denotes the operator of convolution by~$\phi_\beta$. We here note that the alternative formulation
\[
f\mapsto\norm{T_\gamma f-g}^2+\norm{(I-C_\beta)f}^2
\]
was also considered~\cite{hohage2022mollifier}, in which case~\eqref{variational-mollif} does not have the factor $\widehat\phi_{\beta}$ in the numerator. In both cases, the regularization parameter~$\beta$ can be interpreted as the
{\sl resolution level} of the proposed reconstruction formula.
The closer $\beta$ to zero, the finer the target resolution,
and also the more unstable the solution.

A definite advantage of~\eqref{variational-mollif}
over~\eqref{approx-inv-solution} is that little assumption on
the mollifier~$\phi$ is required for its definition and well-posedness.
Obviously, zeros of~$\widehat\gamma$ impose taking cutoff functions
in the numerator of the filter in~\eqref{approx-inv-solution},
and subsequent limitations in the choice of the regularization parameter.
Moreover, the desired well-posedness imposes a rate of decay of the
Fourier transform of the mollifier that guarantees boundedness of
the filter. None of these limitations are necessary in the
variational form of mollification since, by construction, the
denominator remains bounded away from zero.

Formally, we may regard the {\sl filter} in~\eqref{variational-mollif}
as the {\sl symbol} of some pseudo-differential
operator\footnote{Strictly
speaking, the filter $p(\xi)$ does not have
to be smooth, and even if smooth, its derivatives need not satisfy bounds of the form $|\partial^{\alpha}p(\xi)|\le C_{\alpha}(1+|\xi|)^{|\alpha|}$, $\xi\in\Re^n$, $\alpha\in\Ne_0^n$, so the resulting operator is not necessarily
a classical pseudo-differential operator in the sense of H\"ormander~\cite{HIII}.}
applied to~$g$ for reconstructing a mollified version of~$f$,
that shows no dependence on the space variable~$x$.
Due to the {\sl quasi local} nature of convolution operators,
it is only natural to consider extensions of this symbol
that would encode spatial variations of the resolution level.
We therefore consider reconstruction formul\ae{} of the general form
\begin{equation}
\label{variable-resolution}
R_{\alpha,\beta} g(x)=
\int_{\xi\in\Re^n}e^{2\pi ix\xi}\frac{\overline{\widehat{\gamma}(\xi)}\widehat{\phi}(\alpha\beta(x)\xi)}
{\mod{\widehat{\gamma}(\xi)}^2+\mod{1-\widehat{\phi}(\alpha\beta(x)\xi)}^2}Fg(\xi)\ud \xi,\quad x\in\Re^n.
\end{equation}
The above reconstruction formula is a {\sl variable resolution}
version of the one considered in~\cite{hohage2022mollifier}, namely
\begin{equation}
R_{\alpha,\beta} g(x)=
F^{-1}
\left[\frac{\overline{\widehat{\gamma}(\xi)}\widehat{\phi}(\xi)}
{\mod{\widehat{\gamma}(\xi)}^2+\mod{1-\widehat{\phi}(\xi)}^2}\right]Fg.
\end{equation}
By doing so, we hope to improve the resolution in certain regions
of high interest, even if it means reducing the resolution in regions
of lesser interest, without losing overall stability.
The intuition behind this approach is based on the
observation that stability depends not only on the
size of the domain where the Fourier transform of
the unknown object is synthesized but also on the
size of the support of the object (a variation on a Heisenberg theme):
reaching high frequencies should therefore be possible
on regions of interest of limited size.

Notice that the regularization parameter~$\beta$
has been replaced by a function $\beta(x)$,
an infinite-dimensional
object. For the purpose of consistency and
convergence analysis,
we further replace $\beta(x)$ by $\alpha\beta(x)$ with
$\alpha\in\segoc{0}{1}$. We are now back to the familiar
framework of regularization with 1-dimensional parameter,
and we define our candidate {\sl regularization} of the
pseudo-inverse $T^\dagger$ as the family
$(R_{\alpha,\beta})_{\alpha\in\segoc{0}{1}}$ given by
$$
R_{\alpha,\beta}\eqd
\int_{\xi\in\Re^n}e^{2\pi ix\xi}
\frac{\overline{\widehat{\gamma}(\xi)}
\widehat{\phi}(\alpha\beta(x)\xi)}
{\mod{\widehat{\gamma}(\xi)}^2+\mod{1-\widehat{\phi}(\alpha\beta(x)\xi)}^2}F(\cdot)(\xi)\ud\xi.
$$

The rest of this paper is organized as follows: in Section~\ref{section:theo} we establish that, under suitable assumptions, the reconstruction formula~\eqref{variable-resolution} is a {\sl regularization method}. Specifically, we show that the formula~\eqref{variable-resolution} is well-defined pointwise (Theorem~\ref{thm:well-definedness}), that the reconstruction $R_{\alpha,\beta}g$ from~\eqref{variable-resolution} converges pointwise to the correct $f$ as $\alpha\searrow0$ (Theorem~\ref{thm:convi}), and we furthermore find the associated pointwise and $L^p$ convergence rates for sufficiently regular (but rather general) $f$ and for noisy data $g$ (Theorem~\ref{thm:pointLp}). We round up the theoretical analysis by proving an explicit link between the $L^2$-norm stability vs. resolution trade-off (Theorem~\ref{thm:stabilo-1}) in our reconstruction formula. Then, in Section~\ref{section:numerical} we illustrate our methodology using numerical simulations, and we discuss parameter selection rules. Finally, we provide conclusions and outlook in Section~\ref{section:conclusion}.

\section{Essential properties of the reconstruction formula}\label{section:theo}

\begin{assumption}
Throughout the section, we assume that the function $\beta(x)$ is bounded
from above and bounded away from zero, $c\leq \beta(x)\leq C
$ for some positive constants $c$ and $C$, and for all $x\in\Re^n$. We furthermore assume that $\widehat{\phi}(\xi)=1$ iff $\xi=0$, and that $\widehat{\gamma}(0)\neq0$. Finally, our standing assumption is that both the target function $f$ and the data function $g$ are in $L^1(\Re^n)\cap L^2(\Re^n)$.
\end{assumption}

By the Riemann-Lebesgue Lemma, $\widehat f$ and $\widehat g$ are bounded and hence well-defined pointwise in $\Re^n$. Also, we see that the denominator
of the filter in~\eqref{variable-resolution} remains
bounded away from zero.
As a matter of fact, $|\widehat{\gamma}(\xi)|^2$ can only
reach zero at points $\xi$ of norm bounded away from
zero, points at which $|1-\widehat{\phi}(\xi)|^2$
itself is bounded away from zero.

The formal definition
\begin{equation}
\label{eqn:falpha}
R_{\alpha,\beta}g(x):=\int_{\xi\in\Re^n}e^{2\pi i x\xi}\frac{\overline{\widehat{\gamma}}(\xi)\widehat{\phi}(\alpha\beta(x)\xi)}{|\widehat{\gamma}(\xi)|^2+|1-\widehat{\phi}(\alpha\beta(x)\xi)|^2}\widehat{g}(\xi)\ud\xi,\quad x\in\Re^n,
\end{equation}
of the reconstruction formula, given in Section~\ref{section:introduction}, is meaningful pointwise in a rather general regularity setting for the convolution kernel and the mollifier.
\begin{theorem}[\textbf{Pointwise well-definedness of~\eqref{variable-resolution}}]\label{thm:well-definedness}
If
    \begin{enumerate}
        \item $\gamma\in L^2(\Re^n)$ and $\phi\in L^1(\Re^n)$, or if
        \item $\widehat\gamma(\xi)$ is well-defined a.e. in $\Re^n$ and $\phi\in L^2(\Re^n)$,  
    \end{enumerate}
then $R_{\alpha,\beta}g(x)$ in~\eqref{eqn:falpha} is well-defined for all $\alpha>0$ and $x\in\Re^n$.
\end{theorem}
\begin{proof}
If $\widehat\gamma(\xi)$ and $\widehat\phi(\xi)$ are well-defined a.e. in $\Re^n$ then, for each $x\in\Re^n$, and for $j=1,2$, there are positive constants $C_j$ satisfying
\[
\frac{|\widehat\gamma(\xi)|^j}{|\widehat\gamma(\xi)|^2+|1-\widehat\phi(\alpha\beta(x)\xi)|^2}\le C_j,\quad\xi\in\Re^n.
\]
In case (1), the Riemann-Lebesgue lemma implies that $\widehat\phi\in L^{\infty}(\Re^n)$, and then by the Cauchy-Schwartz inequality, for every $x\in\Re^n$, we have $|R_{\alpha,\beta}g(x)|\le C_1\|\widehat\phi\|_{L^{\infty}(\Re^n)}\|\widehat\gamma\|_{L^2(\Re^n)}\|\widehat f\|_{L^2(\Re^n)}$. In case (2), the Cauchy-Schwartz inequality gives $|R_{\alpha,\beta}g(x)|\le C_2\|\widehat\phi\|_2\|\widehat f\|_2$ for each $x\in\Re^n$.
\end{proof}
\begin{remark}
Special cases satisfying the assumptions of Theorem~\ref{thm:well-definedness} include mollifiers in  $C_0^{\infty}(\Re^n)$, which is often the case in practice, and convolution kernels $\gamma\in\mathcal E'(\Re^n)$ (which by the Paley-Wiener-Schwartz theorem~\cite[Theorem 7.3.1, p. 181]{hormander2003analysis} implies $\widehat\gamma\in C^{\infty}(\Re^n)$), as well as $\gamma\in L^p(\Re^n)$ with $p\in[1,2]$ (which by the Hausdorff-Young theorem~\cite[Theorem 7.1.13, p. 165]{hormander2003analysis} implies $\widehat{\gamma}\in L^q(\Re^n)$ with $1/p+1/q=1$).
\end{remark}
In fact, the operator $R_{\alpha,\beta}$ provides a pointwise reconstruction of $f$:
\begin{theorem}[\textbf{Pointwise convergence of~\eqref{variable-resolution}}]\label{thm:convi}
Under the assumptions on $\gamma,\phi$ from Theorem~\ref{thm:well-definedness}, and for pointwise well-defined $f$, we have 
\[
\lim_{\alpha\searrow0}R_{\alpha,\beta}g(x)=f(x),\quad x\in\Re^n.
\]
\end{theorem}
\begin{proof}
We have already shown in Theorem~\ref{thm:well-definedness} that the modulus of the integrand in~\eqref{eqn:falpha} is bounded by an integrable function. The Theorem now follows from Lebesgue's dominated convergence theorem once we note that, for each $\xi\in\Re^n$,
\begin{equation*}
\lim_{\alpha\searrow0}\frac{|\widehat\gamma(\xi)|^2\widehat\phi(\alpha\beta(x)\xi)\widehat f(\xi)}{|\widehat\gamma(\xi)|^2+|1-\widehat\phi(\alpha\beta(x)\xi)|^2}=\widehat f(\xi),
\end{equation*}
as well as that
\[
|f(x)-R_{\alpha,\beta}g(x)|\le\int_{\xi\in\Re^n}\left|\frac{|\widehat\gamma(\xi)|^2\widehat\phi(\alpha\beta(x)\xi)\widehat f(\xi)}{|\widehat\gamma(\xi)|^2+|1-\widehat\phi(\alpha\beta(x)\xi)|^2}-\widehat f(\xi)\right|\ud\xi,\quad x\in\Re^n.\qedhere
\]
\end{proof}


In the following we write $\cb{\xi}=1+|\xi|$
for $\xi\in\Re^n$, and we write $A^s(\Re^n)$, with $s\in\Re$, for the space of functions $u$ for which $|\xi|^s\widehat u\in L^1(\Re^n)$.
\begin{remark}\label{rem:xxi}
Recall that, when $\sigma<-n$,
\begin{eqnarray*}
\int_{\xi\in\Re^n}\xxi^{\sigma}\ud\xi&=\int_{\omega\in S^{n-1}}d\omega\int_{r=0}^{\infty}r^{n-1}(1+r)^{\sigma}dr\le{\rm Area}(S^{n-1})\int_{r=0}^{\infty}(1+r)^{\sigma+n-1}dr\\&=2\pi^{n/2}/(\Gamma(n/2)|\sigma+n|)<\infty,
\end{eqnarray*}
where we use the expression~\cite[Eq. (A.4), p. 306]{TaylorI} for the area of $S^{n-1}$.

We now note that $H^{s'}(\Re^n)\subset A^s(\Re^n)$ when $s'>s+n/2$, since then
\begin{eqnarray*}
    \int_{\Re^n}|\xi|^s|\widehat f(\xi)|\ud\xi&\le&\int_{\Re^n}\xxi^{s'}|\widehat f(\xi)|\xxi^{s-s'}\ud\xi\\&\le&\|\xxi^{s'}\widehat f\|_{L^2(\Re^n)}\sqrt{\int_{\xxi\in\Re^n}\xxi^{2(s-s')}\ud\xi}<\infty.
\end{eqnarray*}
Specifically, we have
\begin{itemize}
    \item $C_0^{\infty}(\Re^n)\subset\mathcal S(\Re^n)\subset H^{\infty}(\Re^n)\subset A^s(\Re^n)$ for all real $s$,
    \item $L^2(\Re^n)=H^0(\Re^n)\subset A^{-n/2-\varepsilon}(\Re^n)$ for all $\varepsilon>0$,
    \item $H^{s'}(\Re^n)\subset A^s(\Re^n)\cap C(\Re^n)$ for $s'>s+n/2$ and any $s>0$,
    \item $A^s(\Re^n)\subset L^{\infty}(\Re^n)$ for any $s\ge0$.
\end{itemize}
\end{remark}
We are now ready to state and prove our main convergence rate results. These are valid also in the presence of suitable spectrally bounded noise in the data. Our proof of the convergence rates uses deterministic integrals to deal with the noise terms, i.e., we are treating the noise on a realization-to-realization basis.
\begin{theorem}[\textbf{Pointwise and $L^p$ convergence rates of~\eqref{variable-resolution}}]\label{thm:pointLp}
If~$\gamma$ and~$\phi$ satisfy the assumptions of Theorem~\ref{thm:well-definedness}, and if constants $a\ge1$, $b,c,d>0$ satisfy $a^{-1}\xxi^{-b}\le|\widehat\gamma(\xi)|$ and $|1-\widehat\phi(\xi)|\le c|\xi|^d$ for $\xi\in\Re^n$, then for all $f\in A^{2(b+d)}(\Re^n)$ and all noise terms $\delta$ with $|\widehat\delta(\xi)|\le E\xxi^{\sigma}$ for some $\sigma<-n-b-d$, we have
\begin{enumerate}
    \item $|f(x)-f^{(\delta)}_{\alpha,\beta}(x)|=O((\alpha\beta(x))^d)$ as $\alpha\rightarrow0$, uniformly in $x\in\Re^n$, and
    \item $\|f-f^{(\delta)}_{\alpha,\beta}\|^p_{L^p(\Re^n)}\lesssim\|(\alpha\beta)^d\|^p_{L^p(\Re^n)}$ for $\alpha\in(0,1)$ and $p\in[1,\infty]$.
\end{enumerate}
\end{theorem}
\begin{proof}
For all $\alpha\in(0,1)$ and $x\in\Re^n$, we have
\begin{eqnarray*}
|f(x)-f_{\alpha,\beta}^{(\delta)}(x)|&=&\Biggl|\int_{\xi\in\Re^n}e^{2\pi ix\xi}\widehat{f}(\xi)\left(1-\frac{|\widehat\gamma(\xi)|^2\widehat{\phi}(\alpha\beta(x)\xi)}{|\widehat\gamma(\xi)|^2+|1-\widehat\phi(\alpha\beta(x)\xi)|^2}\right)\ud\xi\Biggr.\\&+&\Biggl.\int_{\xi\in\Re^n}e^{2\pi ix\xi}\frac{\overline{\widehat{\gamma}}(\xi)\widehat\phi(\alpha\beta(x)\xi)}{|\widehat\gamma(\xi)|^2+|1-\widehat\phi(\alpha\beta(x)\xi)|^2}\widehat\delta(\xi)\ud\xi\Biggr|\\&\le&\int_{\xi\in\Re^n}|\widehat f(\xi)||1-\widehat\phi(\alpha\beta(x)\xi)|\left(1+\frac{|1-\widehat\phi(\alpha\beta(x)\xi)|}{|\widehat\gamma(\xi)|^2}\right)\ud\xi\\&+&\int_{\xi\in\Re^n}\frac{|\widehat\phi(\alpha\beta(x)\xi)|}{|\widehat\gamma(\xi)|}|\widehat\delta(\xi)|\ud\xi\\&\le&c(\alpha\beta(x))^d\int_{\xi\in\Re^n}|\xi|^d|\widehat f(\xi)|\left(1+a^2c(\alpha\beta(x))^d|\xi|^d\xxi^{2b}\right)\ud\xi\\&+&acE(\alpha\beta(x))^d\int_{\xi\in\Re^n}\xxi^{b+d+\sigma}\ud\xi\\&\le& (\alpha\beta(x))^d\Bigl(c\left(1+a^2c\|\beta\|_{L^{\infty}(\Re^n)}^d\right)\int_{\xi\in\Re^n}\xxi^{2(b+d)}|\widehat f(\xi)|\ud\xi\Bigr.\\&+&\Bigr.acE\frac{2\pi^{n/2}}{\Gamma(n/2)|n+b+d+\sigma|}\Bigr).
\end{eqnarray*}
\end{proof}


We next investigate the trade-off between the highest chosen local resolution and the resulting global $L^2$-norm stability of the solution of the deconvolution problem, under mild assumptions on the convolution kernel $\gamma$ and the mollifier $\phi$. Thus, let the Fourier transform of the mollifier $\phi$ be a radial function, $\widehat\phi(\xi)=\Phi(|\xi|)$ for $\xi\in\Re^n$, and furthermore such that
\[
\Phi(t)\le c(1+t)^{-d},\quad t\ge0,
\]
for some positive constants $c$ and $d$. We also assume that the Fourier transform of the convolution kernel $\gamma$ satisfies
\[
a^{-1}\cb{\xi}^{-b}\le|\widehat\gamma(\xi)|\le a\cb{\xi}^b,\quad\xi\in\Re^n,
\]
for some constants $a\ge1$ and $b>0$. From~\eqref{variable-resolution} it follows that the mapping $g\mapsto f_{\alpha,\beta}$ can be written
\[
f_{\alpha,\beta}(x)=\int_{\xi\in\Re^n}e^{2\pi ix\xi}p_{\alpha,\beta}(x,\xi)\widehat{g}(\xi)\ud\xi,\quad x\in\Re^n,
\]
with the 'symbol'
\[
p_{\alpha,\beta}(x,\xi)=\frac{\overline{\widehat{\gamma}(\xi)}\Phi(\alpha\beta(x)|\xi|)}{|\widehat\gamma(\xi)|^2+(1-\Phi(\alpha\beta(x)|\xi|))^2},\quad x\in\Re^n,\,\,\xi\in\Re^n,
\]
and with the corresponding Schwartz kernel
\[
k_{\alpha,\beta}(x,y)=\int_{\xi\in\Re^n}e^{2\pi i\xi(x-y)}p_{\alpha,\beta}(x,\xi)\ud\xi,\quad x,y\in\Re^n.
\]
Assume in the following that the data $g$ are supported in a subset $\Omega\subseteq\Re^n$, and that we restrict the solution $f_{\alpha,\beta}$ to $\Omega$; then the Schwartz kernel $k_{\alpha,\beta}$ of the mapping $g\mapsto f_{\alpha,\beta}$ is defined for $(x,y)\in\Omega^2$. Write $f_{\alpha,\beta,j}$ for the reconstructions of $f$ from the data $g_1$ and $g_2$, respectively.
\begin{theorem}[\textbf{Resolution vs. stability trade-off}]\label{thm:stabilo-1}
There is a positive constant~$C$ that depends precisely on $a$, $b$, $c$, $d$, $n$, and the size $\mu(\Omega)=\int_{\Omega}\ud x$ of the domain $\Omega$, such that 
\[
\|f_{\alpha,\beta,1}-f_{\alpha,\beta,2}\|_{L^p(\Omega)}\le C\left(\alpha\inf_{x\in\Omega}\beta(x)\right)^{-n-b}\| g_1- g_2\|_{L^p(\Omega)},\quad p\in[1,\infty],
\]
as $\alpha\inf_{x\in\Omega}\beta(x)\searrow0$.
\end{theorem}

\begin{proof}
For brevity, write $B=\alpha\inf_{x\in\Omega}\beta(x)$. We first note that
\begin{gather*}
\frac{1}{{\rm Area}(S^{n-1})}\int_{\xi\in\Re^n}\cb{\xi}^{b}(1+B|\xi|)^{-d}\ud\xi=\int_{r=0}^{\infty}r^{n-1}(1+r)^{b}(1+Br)^{-d}dr\\=\frac{\Gamma(d-b-n)\Gamma(n-d)}{\Gamma(-b)}B^{-d}{}_2F_1(d,d-b-n;1-n+d;B^{-1})\\+\frac{(n-1)!\Gamma(d-n)}{\Gamma(d)}B^{-n}{}_2F_1(n,-b;1+n-d;B^{-1}),
\end{gather*}
and then recall the well-known fact that
\[
{}_2F_1(a,b;c;z)\sim\frac{\Gamma(b-a)\Gamma(c)}{\Gamma(b)\Gamma(c-a)}(-z)^{-a}+\frac{\Gamma(a-b)\Gamma(c)}{\Gamma(a)\Gamma(c-b)}(-z)^{-b}
\]
as $z\rightarrow\infty$. There is thus a constant $C'$ dependent only on $b$, $d$, and $n$, and satisfying
\begin{eqnarray*}
\lefteqn{\int_{\xi\in\Re^n}\cb{\xi}^{b}(1+B|\xi|)^{-d}\ud\xi}\\
&\le&
C'{\rm Area}(S^{n-1})\\
&\times&
\Bigl((-1)^{-d}\frac{\Gamma(n-d)\Gamma(-b-n)\Gamma(1-n+d)}{\Gamma(-b)\Gamma(1-n)}\Bigr.\\
&+&
B^{-n-b}(-1)^{b+n-d}\frac{\Gamma(d-b-n)\Gamma(n-d)\Gamma(b+n)\Gamma(1-n+d)}{\Gamma(-b)\Gamma(d)\Gamma(1+b)}\\
&+&
(-1)^n\frac{(n-1)!\Gamma(d-n)\Gamma(-b-n)\Gamma(1+n-d)}{\Gamma(-b)\Gamma(d)\Gamma(1-d)}\\
&+&
\Bigl.B^{-n-b}(-1)^{b}\frac{\Gamma(d-n)\Gamma(b+n)\Gamma(1+n-d)}{\Gamma(d)\Gamma(1+n-d+b)}\Bigr)\\
&=&
O(B^{-n-b})\,\,\,{\rm as}\,\,\,B\searrow0.
\end{eqnarray*}
Next, for all $\alpha>0$ and $y\in\Omega$, we have
\begin{eqnarray*}
\int_{x\in\Omega}|k_{\alpha,\beta}(x,y)|\ud x&\le&\int_{x\in\Omega}\int_{\xi\in\Re^n}\frac{|\widehat\gamma(\xi)|\Phi(\alpha\beta(x)|\xi|)}{|\widehat{\gamma}(\xi)|^2+(1-\Phi(\alpha\beta(x)|\xi|))^2}\ud\xi \ud x\\&\le&\mu(\Omega)a^3c\int_{\xi\in\Re^n}\cb{\xi}^{b}(1+B|\xi|)^{-d}\ud\xi,
\end{eqnarray*}
and similarly, for all $\alpha>0$ and $x\in\Omega$,
\[
\int_{y\in\Omega}|k_{\alpha,\beta}(x,y)|dy\le\mu(\Omega)a^3c\int_{\xi\in\Re^n}\cb{\xi}^{b}(1+B|\xi|)^{-d}\ud\xi.
\]
Consequently, by~\cite[Proposition 5.1, p. 573]{TaylorI},
\[
\|g\mapsto f_{\alpha,\beta}\|_{L^{\tau}(\Omega)\rightarrow L^{\tau}(\Omega)}\le C'\mu(\Omega)a^3c(\alpha\inf_{z\in\Omega}\beta(z))^{-n-b}
\]
for all $\tau\in[1,\infty]$ and for sufficiently small positive $\alpha\inf_{x\in\Omega}\beta(x)$. \qedhere
\end{proof}

\section{Numerical examples, numerical stability analysis, additional constraints}\label{section:numerical}

\subsection{Parameter selection}
\label{sec:icn}

A central issue in inverse problems is stability, which refers to the sensitivity of the solution with respect to perturbations in the data. In the linear setting, stability can be assessed in different but complementary ways. One classical approach is to examine the norm of the reconstruction operator, which quantifies the amplification of absolute errors from the data to the solution. Alternatively, the condition number of the operator governs the propagation of relative errors. A large condition number indicates that small relative perturbations in the data can lead to significant relative errors in the solution. These notions of stability are key when choosing an appropriate regularization strategy, particularly in the determination of regularization parameters, which must balance stability with fidelity to the observed data.

In this work, we adopt a stability-driven approach to regularization. Rather than selecting the regularization parameter based solely on fidelity or resolution criteria, we begin by fixing a prescribed upper bound on the instability level, quantified either by the norm of the reconstruction operator or by the condition number, depending on the context. Within this constraint, we then seek the best achievable average resolution compatible with the desired stability. Once this global balance is established, we explore the possibility of enhancing the local resolution in a specific region of interest. The underlying idea is that, by allowing a controlled degradation of resolution in less critical areas, one can improve the reconstruction quality where it matters most, without violating the global stability constraint.

In the next subsections, we address deblurring problems and rely on standard tools from numerical analysis to assess the stability of our reconstruction formula. As a first step, we consider a uniform resolution across the entire domain, and determine a suitable average level of resolution by fixing an admissible threshold for the stability measure.
Building on this global analysis, we then manually identify a small region of interest and allow $\beta$ to vary spatially across the image. By concentrating resolution in this selected area and reducing it elsewhere, we demonstrate that localized enhancement can be achieved without exceeding the predefined global stability bound.

For a given reconstruction, starting with uniform $\beta$, we should be able to improve local resolution with same global stability.

\subsection{Deconvolution in dimension one}
\begin{figure}[ht]
    \centering
    \includegraphics[scale=0.5]{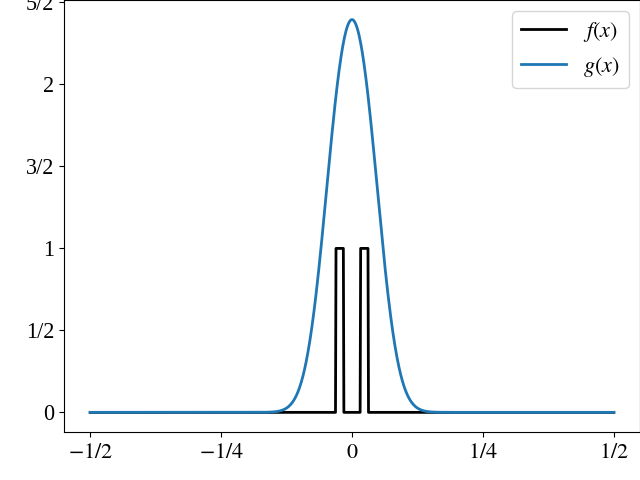}
    \caption{The ground truth $f(x)$ and the noiseless data $g(x)=\gamma\ast f(x)$.}
    \label{fig:fg}
\end{figure}

\begin{figure}[ht]
    \centering
    \includegraphics[width=0.49\linewidth]{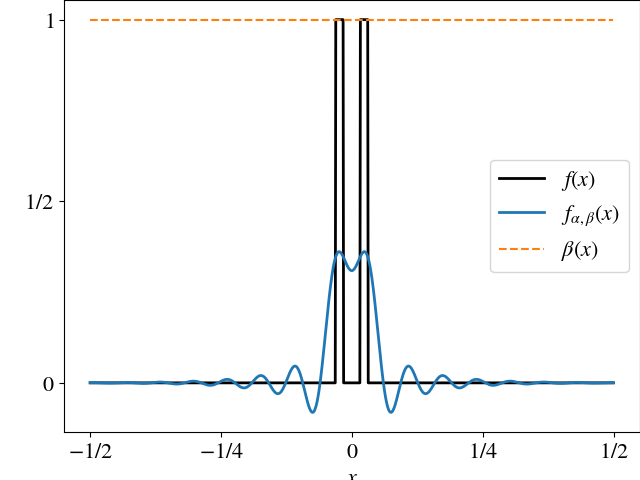}
\includegraphics[width=0.49\linewidth]{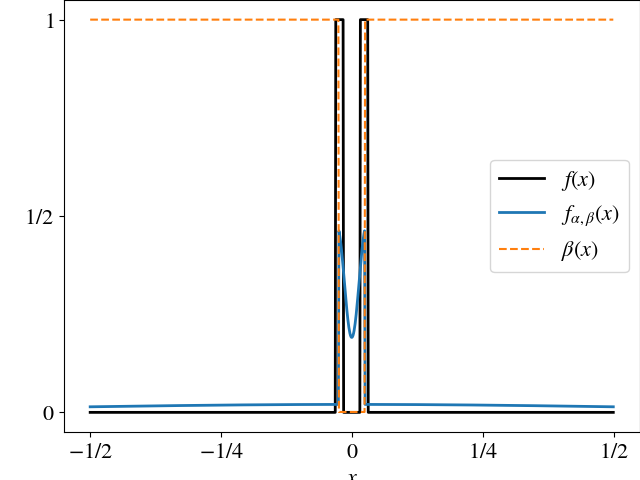}\\
\includegraphics[width=0.49\linewidth]{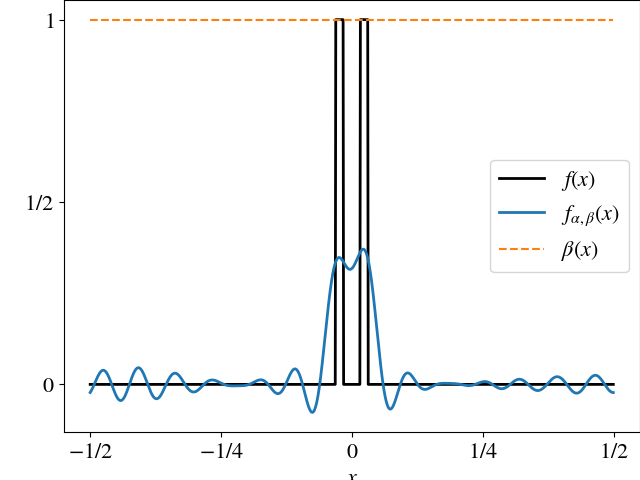}
\includegraphics[width=0.49\linewidth]{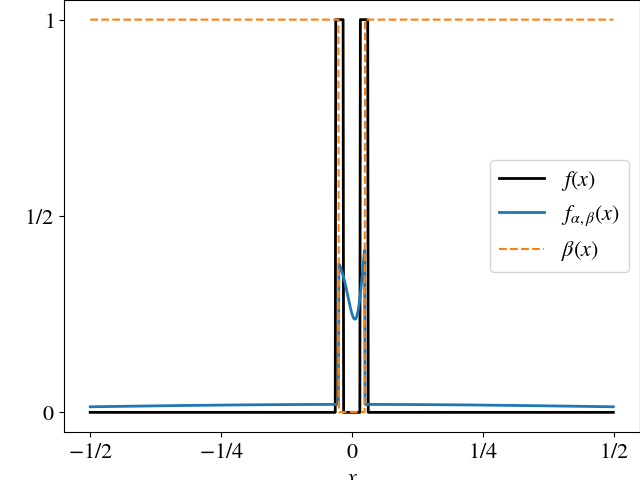}
    \caption{Deconvolution using $\alpha=10^{-3}$ and constant-valued $\beta$ (left column) vs. using $\alpha=0.262$ and single-dip $\beta$ (right column). Estimated condition number is approximately $7.6$ in all cases. Noise level in data $10^{-4}$ (top row) and $5\cdot10^{-2}$ (bottom row).}
    \label{fig:res1D}
\end{figure}
We now give numerical examples of constant-resolution and variable-resolution deconvolution in dimension one. Figure~\ref{fig:fg} shows the graphs of our chosen target (ground truth) piecewise constant function with 'pulses' of 1/64 width and 1/32 apart edge-to-edge, and of the noiseless data $g(x)=\gamma\ast f(x)$, $x\in\Re$, where we use the Gaussian convolution kernel $\gamma(x)=0.1\exp(-x^2/0.05^2)$, $x\in\Re$. Figure~\ref{fig:res1D} shows the results of the application of our reconstruction formula~\eqref{variable-resolution} to noiseless and noisy data $g(x)$ for various choices of~$\alpha$ and $\beta(x)$.

Our variable-resolution deconvolution reconstructs the localized high-frequency features of the target function $f$ better than the classical deconvolution, while maintaining the global stability.

\subsection{Deconvolution in dimension two}

Figures \ref{fig:f_and_gnoisy} and \ref{fig:reconstructions} illustrate the reconstruction of a modified $601 \times 601$ blurred and noisy Shepp-Logan phantom with a piecewise constant $\beta$. A region of interest with a low value of $\beta$ was chosen around 2 blobs with high intensity. The numerical example was set up in MATLAB. The ground truth image was blurred with a Gaussian filter
\[
\gamma(x,y) = \frac{1}{2\pi\sigma^{2}}
\exp\!\left(-\,\frac{x^{2}+y^{2}}{2\sigma^{2}}\right)
\] with $\sigma=7$
and Gaussian i.i.d. noise was added to the blurred image. The reconstruction was carried out using a discretization of formula \eqref{eqn:falpha}. Fourier transforms of the functions~$g$, $\gamma$ and $\phi$ were computed using the Fast Fourier Transform (FFT). Let $\bm{{g}}\in \Re^{601\times 601}$ be the values of the function~$g$ evaluated on a pixel grid. Let $\bm{\widehat{g}}$ be the FFT of $\bm{g}$. We calculate the Fourier transforms of the functions $\phi$ and $\gamma$ analytically. The Fourier transform of~$\gamma$ can be evaluated on the same grid as $\bm{\widehat{g}}$ giving us $\bm{\widehat{\gamma}}$. For each value of $\beta(x)$, we can evaluate $\widehat{\phi}(\alpha \beta(x)\xi)$ on that same grid, giving us $\bm{\widehat{\phi}(\alpha\beta(x))}$.

The reconstruction formula \eqref{eqn:falpha} is discretized using the low order scheme
\begin{equation*}
    (R_{\alpha, \beta}g)(x_1,x_2) = \sum_{j,k=1}^{601}  e^{2\pi i (x_1\xi_{1j}+x_2\xi_{2k})} \frac{\overline{\bm{\widehat{\gamma}}}_{jk}  \bm{\widehat{\phi}(\alpha\beta(x))}_{jk}}{|\bm{\widehat{\gamma}}_{jk}|^2 + |1-\bm{\widehat{\phi}(\alpha\beta(x))}_{jk}|^2} \frac{1}{(\Delta \xi)^2}
\end{equation*}

Using the right choice of grids, we can evaluate this sum using a partial inverse FFT. Letting $\beta$ be piecewise constant with $L$ different values, $\beta_1, \hdots, \beta_L$, we get exactly $L$ arrays $\bm{\widehat{\phi}(\alpha\beta_l})$, $l=1,\hdots, L$. Thus the reconstruction can be computed $L$ different partial inverse FFTs.
In the figure below, we consider 
\[
\beta(x)= \begin{cases}
    \beta_1, \quad &{x \in \Omega_1, }\\
    \beta_2, \quad &{x \in \Omega_2, }
\end{cases}
\]
where $\Omega_1$ is the region of interest.
Let $g_n = g + \varepsilon$, where $\varepsilon \sim \mathcal{N}(0,\sigma^2)$ with $\sigma = 10^{-2}$, and let $\delta g = g_n - g$ denote the (small) perturbation in the data. We compute
\[
c_1 = \frac{\|\delta g\|}{\|g\|}.
\]
Using our reconstruction operator, let $f_{\mathrm{rec}0}$ be the reconstruction from $g$ and $f_{\mathrm{rec}}$ the reconstruction from $g_n$. Define $\delta f = f_{\mathrm{rec}} - f_{\rm rec0}$, and compute
\[
c_2 = \frac{\|\delta f\|}{\|f_{\rm rec0}\|}.
\]
An approximate lower bound on the condition number $\kappa$ (the “global stability indicator”) is then
\[
\kappa \approx \frac{c_2}{c_1}.
\]
We run the code twice: once with constant $\beta$, obtaining $\kappa \approx 0.06$, and once with variable $\beta$, also obtaining $\kappa \approx 0.06$.

\begin{figure}[ht]
    \centering
    \begin{subfigure}{0.48\textwidth}
        \centering
        \includegraphics[width=\linewidth]{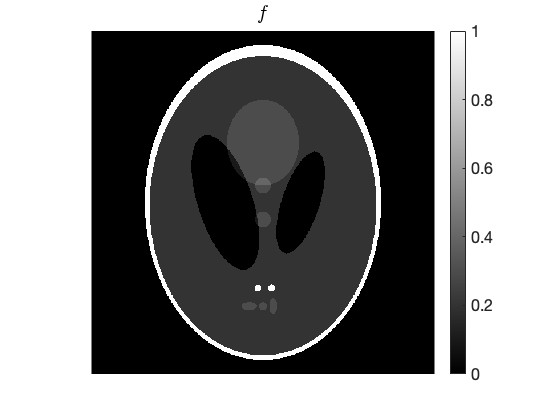}
        \caption{Ground truth $f$.}
    \end{subfigure}
    \begin{subfigure}{0.48\textwidth}
        \centering
        \includegraphics[width=\linewidth]{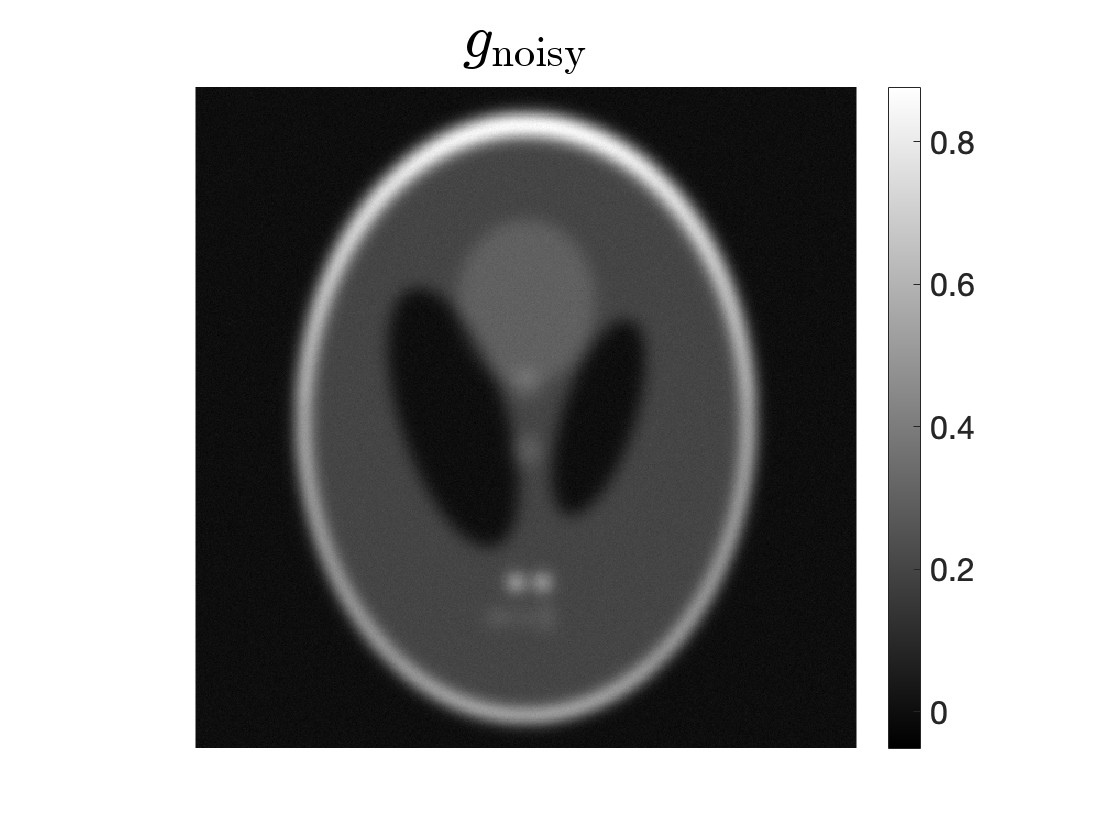}
        \caption{Noisy blurred data $g_{\rm noisy}=\gamma\ast f+\varepsilon$.}
    \end{subfigure}
    \caption{Ground truth image and corresponding noisy blurred data.}
    \label{fig:f_and_gnoisy}
\end{figure}

\begin{figure}[ht]
    \centering
    \begin{subfigure}{0.48\textwidth}
        \centering
        \includegraphics[width=\linewidth]{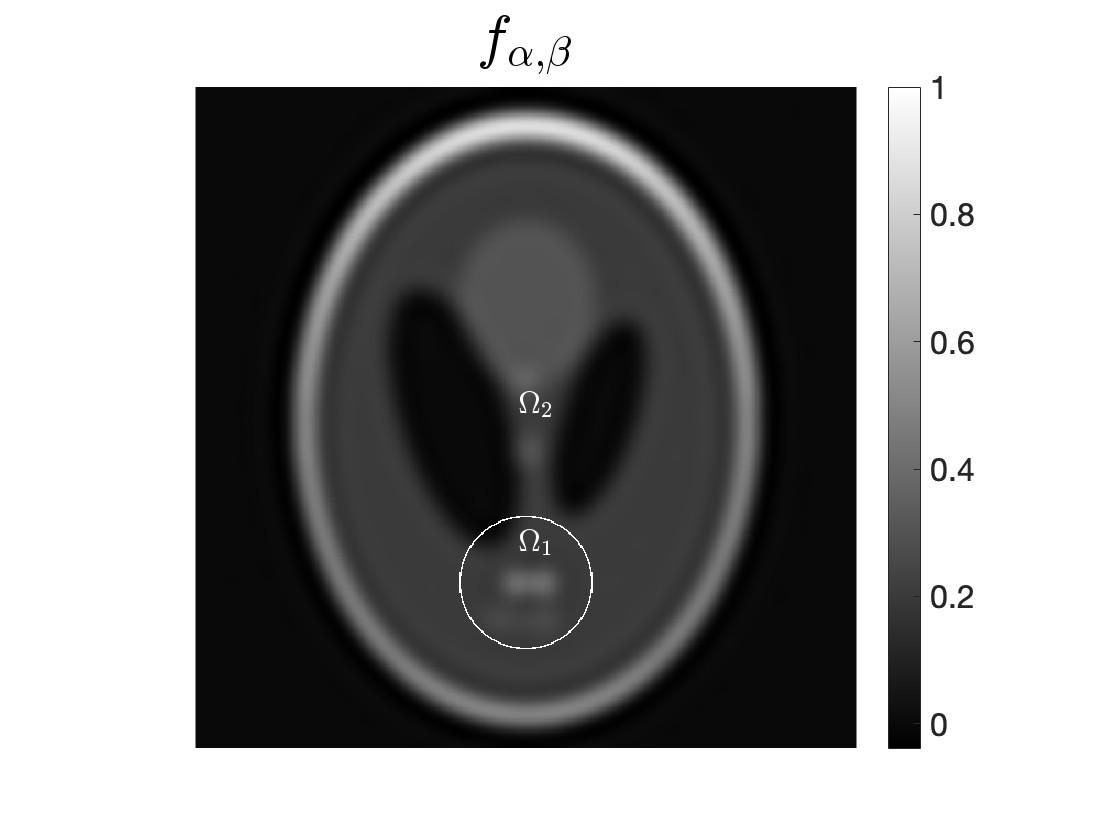}
   \caption{}
    \end{subfigure}
    \begin{subfigure}{0.48\textwidth}
        \centering
        \includegraphics[width=\linewidth]{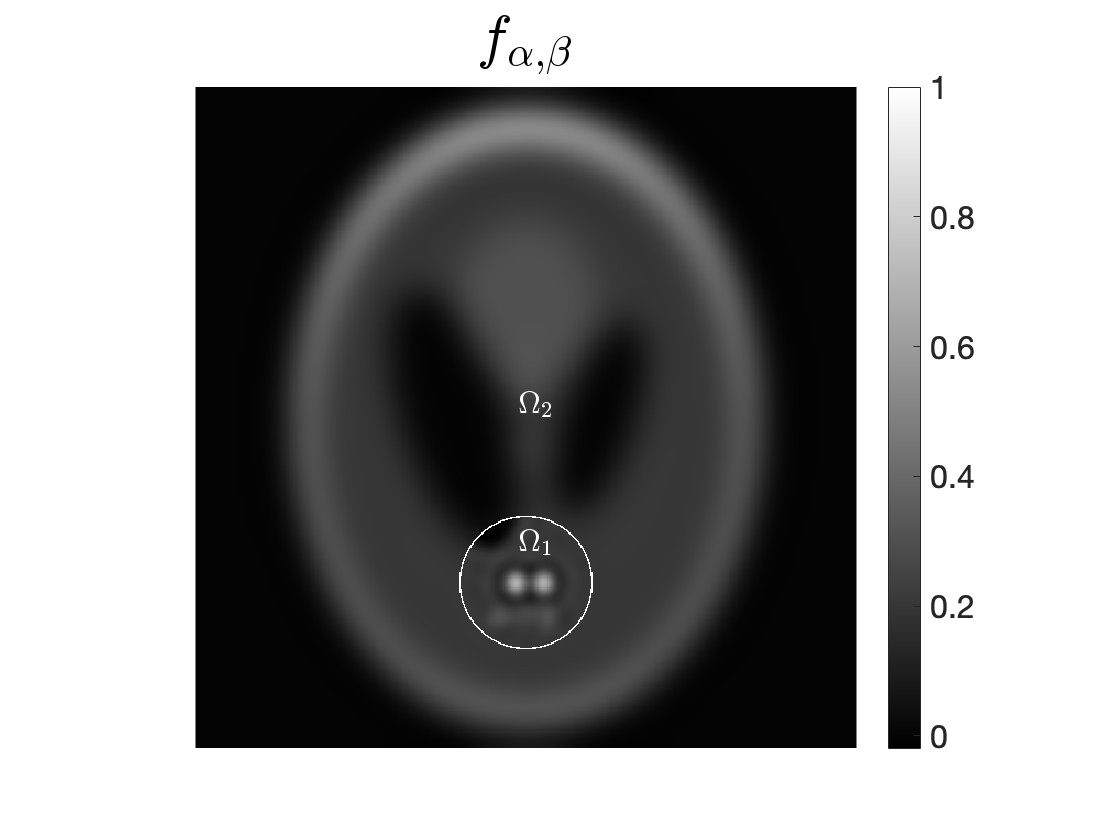}
        \caption{}
    \end{subfigure}
    \caption{(A) Reconstruction $f_{\alpha,\beta}$ with $\beta(x)=\beta_1=2.5$ in $\Omega_1$ and $\beta(x)=\beta_2=2.5$ in $\Omega_2$, having global stability $\kappa \approx 0.06$.\\
             (B) Reconstruction $f_{\alpha,\beta}$ with $\beta(x)=\beta_1=0.8$ in $\Omega_1$ and $\beta(x)=\beta_2=6.4$ in $\Omega_2$, having global stability $\kappa \approx 0.06$.}
    \label{fig:reconstructions}
\end{figure}
\begin{figure}[ht]
    \centering
    
    \begin{subfigure}{0.45\textwidth}
        \centering
        \includegraphics[width=\linewidth]{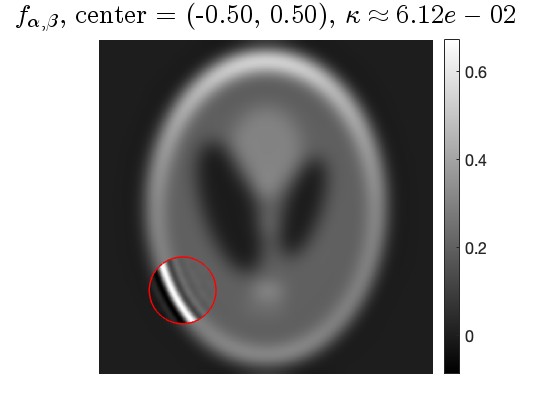}
        \caption{}
    \end{subfigure}
    \begin{subfigure}{0.45\textwidth}
        \centering
        \includegraphics[width=\linewidth]{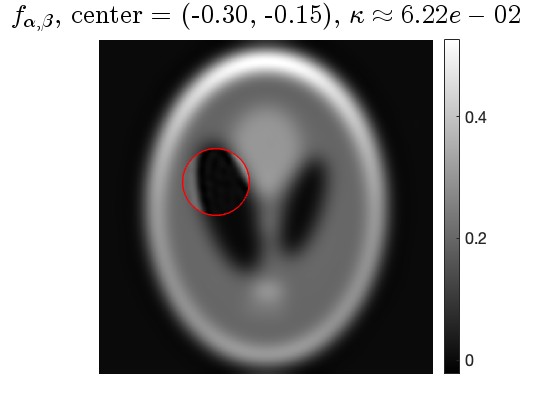}
        \caption{}
    \end{subfigure}
    
    \begin{subfigure}{0.45\textwidth}
        \centering
        \includegraphics[width=\linewidth]{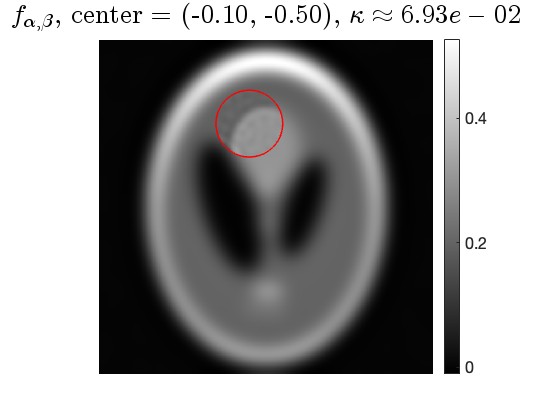}
        \caption{}
    \end{subfigure}
    \begin{subfigure}{0.45\textwidth}
        \centering
        \includegraphics[width=\linewidth]{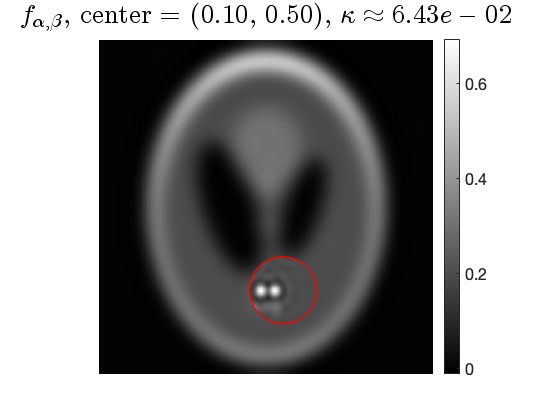}
        \caption{}
    \end{subfigure}
    
    \begin{subfigure}{0.45\textwidth}
        \centering
        \includegraphics[width=\linewidth]{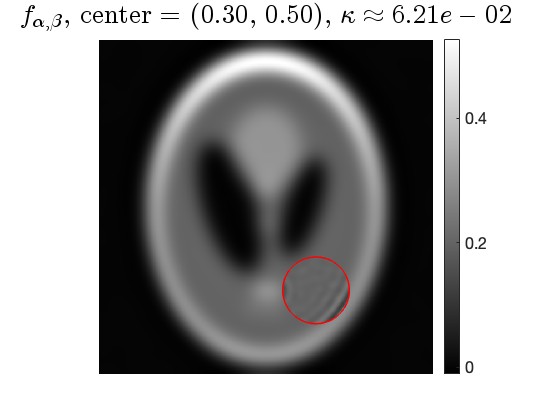}
        \caption{}
    \end{subfigure}
    \begin{subfigure}{0.45\textwidth}
        \centering
        \includegraphics[width=\linewidth]{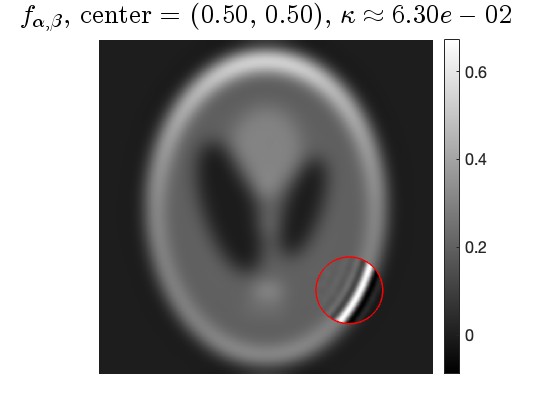}
        \caption{}
    \end{subfigure}
    
    \caption{Reconstructions $f_{\alpha,\beta}$ with different positions of the ROI and corresponding global stability estimate $\kappa$.}
    \label{fig:roi-kappa-snapshots}
\end{figure}
 Figure \ref{fig:roi-kappa-snapshots} shows six panels illustrating the reconstructions $f_{\alpha,\beta}$ for various placements of the circular ROI "$\Omega_1$" in the image domain. Inside $\Omega_1$ we use $\beta(x) = \beta_{\text{low}}$, while in the complement $\Omega_2$ we use $\beta(x) = \beta_{\text{high}}$, so that moving the circle changes where the weaker/stronger regularisation is applied. Each image shows the reconstructed solution together with the current location of $\Omega_1$ (red circle) and the corresponding global stability constant $\kappa$, computed as the ratio between the relative perturbations in the reconstruction and in the data. This allows us to visualise how the spatial position of the ROI affects the overall stability of the inverse problem.

It is natural to ask whether our 'scanning' of the image using multi-resolution deconvolution, as illustrated in Figure~\ref{fig:roi-kappa-snapshots}, can instead be effected simply by a series of constant-resolution reconstructions of the image within the moving small disk. We submit that the latter approach would be disadvantageous, in that the stability of the reconstruction is generally reduced when the size of the support of the object under reconstruction is reduced. Thus, maintaining the same stability $\kappa$ would result in a comparatively poorer resolution of the reconstruction within the small disk, see Fig.~\ref{fig:worse}. In this figure, the small disk is for simplicity exchanged with an equal-area rectangle.
\begin{figure}
\centering
    \includegraphics[scale=0.2]{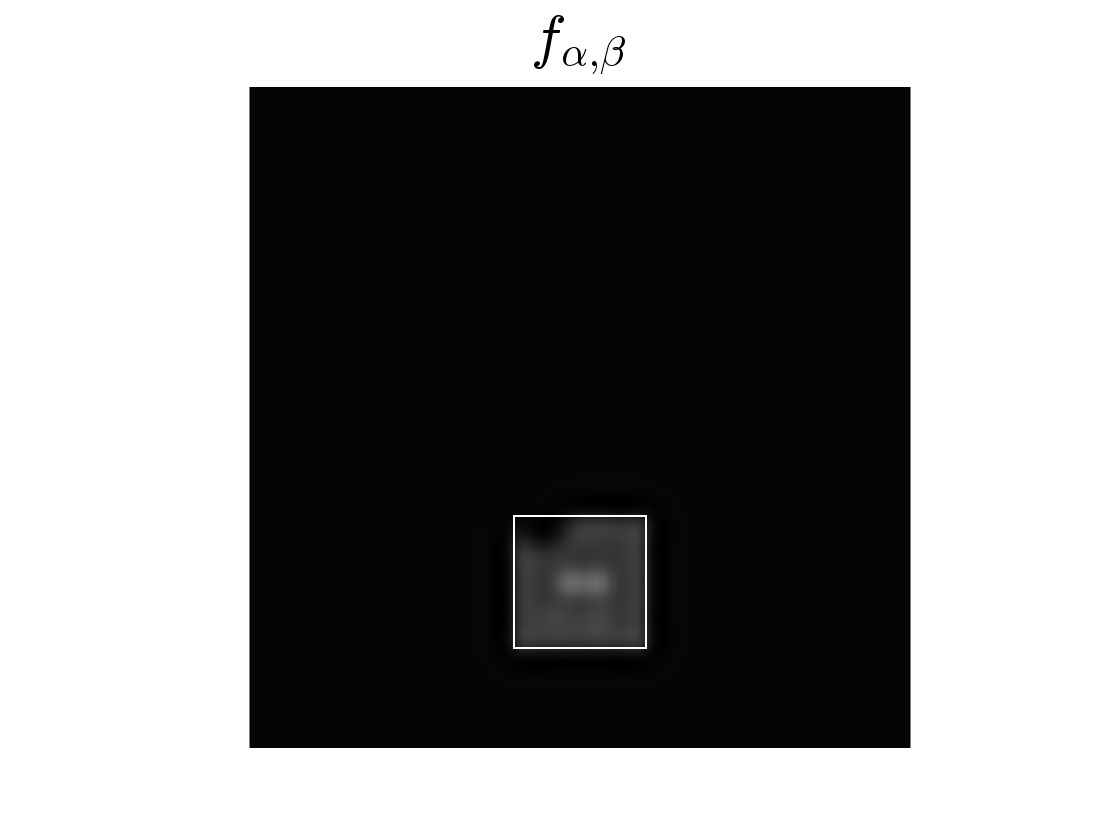}
    \caption{A reconstruction that keeps the stability estimate at $\kappa\approx0.06$ but only takes into account a small piece of the data. Compare the quality with our reconstructions in the small disks in Figure~\ref{fig:roi-kappa-snapshots}.}
    \label{fig:worse}
\end{figure}

\section{Conclusions and outlook}\label{section:conclusion}

We have developed and analyzed a variational mollification framework for deconvolution that admits spatially varying resolution through a pseudodifferential‐like reconstruction operator. Our framework allows for local variation of resolution while keeping the global stability of the deconvolution problem constant. Under mild assumptions on the convolution kernel $\gamma$ and the mollifier $\phi$, we established pointwise convergence, norm convergence, and convergence rate results, allowing for noise in the data. We also obtained an explicit stability–versus–resolution trade-off that ties the global $L^{2}$ stability to the smallest local resolution level. Numerical experiments in one and two dimensions corroborate the theory. Specifically, we demonstrate selective sharpening without degrading global stability. Further work includes a generalization of the allowed regularities of the convolution kernel $\gamma$, the mollifier $\phi$, and the target function $f$, as well as a generalization of the noise statistics, in particular the inclusion of spatially variable noise and a treatment of the contribution of noise in the reconstruction formula in terms of a stochastic integral.

\section*{Acknowledgements}
This work was supported by the Villum Foundation grants no. 58857 and 25893.

\bibliographystyle{siam}
\bibliography{variable-resolution}{}

@article{bonnefond2009variational,
  title={A variational approach to the inversion of some compact operators},
  author={Bonnefond, Xavier and Mar{\'e}chal, Pierre},
  journal={Pacific journal of optimization},
  volume={5},
  number={1},
  pages={97--110},
  year={2009}
}

@article{alibaud2009variational,
  title={A variational approach to the inversion of truncated {F}ourier operators},
  author={Alibaud, Natha{\"e}l and Mar{\'e}chal, Pierre and Saesor, Yaowaluk},
  journal={Inverse Problems},
  volume={25},
  number={4},
  pages={045002},
  year={2009},
  publisher={IOP Publishing}
}

@book{engl1996regularization,
  title={Regularization of inverse problems},
  author={Engl, Heinz Werner and Hanke, Martin and Neubauer, Andreas},
  volume={375},
  year={1996},
  publisher={Springer Science \& Business Media}
}

@book{HIII,
    title={The Analysis of Linear Partial Differential Operators III: Pseudo-Differential Operators},
    author={H\"ormander, L.},
    series={Grundlehren der mathematischen Wissenschaften},
    number=274,
    year={1985},
    publisher={Springer}
}

@book{TaylorI,
  title={Partial Differential Equations I: Basic Theory},
  author={Taylor, M. E.},
  series={Applied Mathematical Sciences},
  volume={115},
  year={2011},
  edition={Second},
  publisher={Springer}
}

@article{hohage2022mollifier,
  title={A mollifier approach to the deconvolution of probability densities},
  author={Hohage, Thorsten and Mar{\'e}chal, Pierre and Simar, L{\'e}opold and Vanhems, Anne},
  journal={Econometric Theory},
  pages={1--40},
  year={2022},
  publisher={Cambridge University Press}
}

@book{hormander2003analysis,
  title={{The Analysis of Linear Partial Differential Operators I: Distribution Theory and Fourier Analysis}},
  author={H{\"o}rmander, Lars},
  year={2003},
  publisher={Springer}
}

@article{lannes1987stabilized,
  title={Stabilized reconstruction in signal and image processing: I. Partial deconvolution and spectral extrapolation with limited field},
  author={Lannes, Andr{\'e} and Roques, Sylvie and Casanove, Marie-Jos{\'e}},
  journal={Journal of Modern Optics},
  volume={34},
  number={2},
  pages={161--226},
  year={1987},
  publisher={Taylor \& Francis}
}

@article{louis1990mollifier,
  title={A mollifier method for linear operator equations of the first kind},
  author={Louis, Alfred K and Maass, Peter},
  journal={Inverse problems},
  volume={6},
  number={3},
  pages={427},
  year={1990},
  publisher={IOP Publishing}
}

@book{murio2011mollification,
  title={The mollification method and the numerical solution of ill-posed problems},
  author={Murio, Diego A},
  year={2011},
  publisher={John Wiley \& Sons}
}

@book{schuster2007method,
  title={The method of approximate inverse: theory and applications},
  author={Schuster, Thomas},
  volume={1906},
  year={2007},
  publisher={Springer}
}

\end{document}